\documentclass{amsart}

\newtheorem{theorem}{Theorem}[section]
\newtheorem{lemma}[theorem]{Lemma}

\newtheorem{cor}[theorem]{Theorem}

\theoremstyle{definition}

\theoremstyle{remark}
\newtheorem{remark}[theorem]{Remark}

\numberwithin{equation}{section}

\newcommand{\Z}{\mathbb Z}
\newcommand{\C}{\mathbb C}
\newcommand{\F}{\mathbb F}
\newcommand{\Fh}{{}_2F_1}
\newcommand{\Fha}[4]{{}_2F_1\left(\left. \begin{array}{cc} #1 & #2 \\     \  & #3  \end{array}\right| #4 \right)}
\newcommand{\tr}{\textrm{tr}}
\newcommand{\bin}[2]{\left( {#1 \atop #2} \right)}

\begin{document}

\title[Gaussian Hypergeometric Evaluations of Traces of Frobenius]{Gaussian Hypergeometric Evaluations of Traces of Frobenius for Elliptic Curves}

\author{Catherine Lennon}
\address{Department of Mathematics, Massachusetts Institute of Technology, 77 Massachusetts Avenue, Cambridge, MA 02139}
\email{clennon@math.mit.edu}
\thanks{This work was supported by the Department of Defense (DoD) through the National Defense Science \& Engineering Graduate Fellowship (NDSEG) Program.}


\date{\today}

\commby{Ken Ono}

\begin{abstract} We present here a formula for expressing the trace of the Frobenius endomorphism of an elliptic curve $E$ over $\mathbb{F}_q$ satisfying $j(E)\neq 0,1728$ and $q\equiv 1 \pmod{12}$ in terms of special values of Gaussian hypergeometric series.  This paper uses methods introduced in \cite{fuselier} for one parameter families of curves to express the trace of Frobenius of $E$ as a function of its $j$-invariant and discriminant instead of a parameter, which are more  intrinsic characteristics of the curve.
\end{abstract}

\maketitle

\section{Introduction}
Gaussian hypergeometric series  were first defined by Greene in \cite{greene} as finite field analogues of the classical hypergeometric series. Since then, they have been shown to possess interesting arithmetic properties; in particular, special values of these functions can be used to express the number of $\F_p$-points on certain varieties.  For example, results in \cite{koike} and \cite{ono} presented formulas expressing the number of $\F_p$-points of elliptic curves in certain families as special values of Gaussian hypergeometric series. These formulas, however, only used trivial and quadratic characters as parameters, and the task remained to find some expressions with parameters that were characters of higher orders \cite{onoweb}.

Recently in \cite{fuselier}, Fuselier provided formulas for certain families of elliptic curves which involved Gaussian hypergeometric series with characters of order 12 as parameters, under the assumption that $p\equiv 1\pmod{12}$ (which is necessary to assure that characters of order 12 exist). In \cite{lennon}, we provide a formula for the trace  of Frobenius for curves with 3-torsion and $j$-invariant not equal to 0,1728 using characters of order three. Again, we must  assume that $p\equiv 1\pmod{3}$.

In all of the  previous results, the character parameters in the hypergeometric series depended on the family of curves considered. In addition, the values at which the hypergeometric series were evaluated were functions of the coefficients, and so depended  on the model used.  Here, we give a general formula expressing the number of $\F_p$-points of an elliptic curve in terms of more intrinsic properties of the curve.  Consequently, this characterization is coordinate-free and can be used to describe the number of points on any elliptic curve $E(\F_{p^e})$, with $j(E)\neq 0, 1728$ and $p^e\equiv 1 \pmod{12}$ without having to put the curve in a specific form. In particular, the formula holds over $\F_{p^2}$ for all odd $p$ whenever $j\neq 0, 1728$.

Let $q=p^e$ be a power of an odd prime and $\F_q$ the finite field of $q$ elements. Extend each character $\chi \in \widehat{\F^*_q}$ to all of $\F_q$ by setting $\chi(0):=0$.  For any two characters $A,B\in \widehat{\F^*_q}$ one can  define the normalized Jacobi sum by
\begin{equation}
\bin{A}{B}:=\frac{B(-1)}{q}\sum_{x\in \F_q}A(x)\bar{B}(1-x)=\frac{B(-1)}{q} J(A, \bar{B}),
\end{equation}
where $J(A,B)$ denotes the usual Jacobi sum.

Recall the definition of the  \emph{Gaussian hypergeometric series over $\F_q$} first defined by Greene in \cite{greene}. For any positive integer $n$ and characters $A_0,A_1,...,A_n$ and $B_1,B_2,...,B_n \in \widehat{\F_q^*}$, the Gaussian hypergeometric series $_{n+1}F_n$ is defined to be
\begin{equation}
 _{n+1}F_n \left( \left.{\begin{array}{cccc}
                A_0 & A_1 & ... & A_n \\
                \  & B_1 & ...  & B_n
              \end{array}
}\right|x\right)_q:=\frac{q}{q-1}\sum_{\chi \in \widehat{\F_q^*}}\left(A_0\chi \atop \chi\right)\left(A_1 \chi\atop B_1 \chi\right)...\left(A_n \chi \atop B_n \chi \right)\chi(x).
\end{equation}

If we let $a(E(\F_q))$ denote the trace of the Frobenius endomorphism on $E$, then
$$a(E(\F_q))=q+1-|E(\F_q)|$$
and the following theorem expresses this value, and therefore also $|E(\F_q)|$,  in terms of Gaussian hypergeometric series.

\begin{theorem}\label{genweier} Let $q=p^e$, $p>0$ a prime and $q\equiv 1 \pmod{12}$. In addition, let $E$ be an elliptic curve over $\F_q$ with $j(E)\neq 0,1728$ and $T \in \widehat{\F_q^*}$ a generator of the character group. The trace of the Frobenius map on $E$ can be expressed as
\begin{equation}\label{genweiereqn}
a(E(\F_q))=-q\cdot T^{\frac{q-1}{12}}\left(\frac{1728}{\Delta(E)}\right)\cdot\Fha{T^{\frac{q-1}{12}}}{T^{\frac{q-1}{12}}}{T^{\frac{2(q-1)}{3}}}{\frac{j(E)}{1728}}_q,
\end{equation}
where $\Delta(E)$ is the discriminant of $E$.
\end{theorem}
 \begin{remark} When $p\not\equiv 1 \pmod{12}$, information about $a(E(\F_p))$ can still be gained from Theorem \ref{genweier}. Because $p^2 \equiv 1\pmod{12}$ for all $p>3$, Theorem \ref{genweier} applies with $q=p^2$. Using the relationship
 $$a(E(\F_p))^2=a(E(\F_{p^2}))+2p$$
 one can then determine $a(E(\F_p))$ up to a sign.
 \end{remark}

\section{Proof of Theorem \ref{genweier}}

\subsection{Elliptic Curves in Weierstrass Form}
Theorem \ref{genweier} will follow as a consequence  of the next proposition after applying transformation laws for Gaussian hypergeometric series.
Recall that an elliptic curve can be written in Weierstrass form as
$$E: y^2=x^3+ax+b.$$
We may express the trace of Frobenius on the curve above as a special value of a hypergeometric function in the following way.
\begin{theorem} \label{weier}
Let $q=p^e$, $p>3$ a prime and $q\equiv 1\pmod{12}$. Let $E(\F_q)$ be an elliptic curve over $\F_q$ in Weierstrass form with $j(E)\neq 0, 1728$. Then the trace of the Frobenius map on $E$ can be expressed as
$$a(E(\F_q))=-q\cdot T^{\frac{q-1}{4}}\left(\frac{a^3}{27}\right) \cdot
\Fha{ T^{\frac{q-1}{12}} }{ T^{\frac{5(q-1)}{12}} }{ T^{\frac{q-1}{2}}}{-\frac{27b^2}{4a^3}}_q.$$
\end{theorem}
\begin{proof}
The method of this proof follows similarly to that given in \cite{fuselier}.  Let $|E(\F_q)|$ denote the number of projective points of $E$ in $\F_q$.  If we let
$$P(x,y)=x^3+ax+b-y^2,$$
 then  $|E(\F_q)|$ may be expressed as
$$|E(\F_q)|-1=\#\{(x,y)\in \F_q \times \F_q: P(x,y)=0\}.$$
Define the additive character $\theta: \F_q \to \C^*$ by
\begin{equation}
\theta(\alpha)=\zeta^{\tr(\alpha)}
\end{equation}
 where  $\zeta=e^{2\pi i/p}$ and $\tr: \F_q \to \F_p$ is the trace map, ie $\tr(\alpha)=\alpha + \alpha^p+\alpha^{p^2}+...+\alpha^{p^{e-1}}$. As in \cite{fuselier}, we begin by repeatedly using the elementary identity from \cite{ire}
\begin{equation}\label{thetasum}
\sum_{z\in \F_q}\theta(zP(x,y))=\left\{
                                    \begin{array}{ll}
                                      q & \hbox{if $P(x,y)=0$} \\
                                      0 & \hbox{if $P(x,y)\neq0$}
                                    \end{array}
                                  \right.
\end{equation}
to express the number of points as
\begin{align*}
q\cdot (\#E(\F_q)-1)=&\sum_{z\in \F_q}\sum_{x,y \in \F_q}\theta (z P(x,y))\\
=&q^2+\underbrace{\sum_{z \in \F_q^*}\theta(zb)}_A+ \underbrace{\sum_{y,z\in \F_q^*}\theta(zb) \theta(-zy^2)}_B
+\underbrace{\sum_{z,x\in \F_q^*}\theta(zx^3)\theta(zax)\theta(zb)}_C\\
&+\underbrace{\sum_{x,y,z\in \F_q^*}\theta(zP(x,y))}_D.
\end{align*}
We will evaluate each of these labeled terms using the following lemma from \cite{fuselier}.

\begin{lemma}[\cite{fuselier}, Lemma 3.3]\label{theta} For all $\alpha \in \F_q^*$,
$$\theta(\alpha)=\frac{1}{q-1}\sum_{m=0}^{q-2}G_{-m}T^m(\alpha),$$ where $T$ is a fixed generator of the character group and $G_{-m}$ is the Gauss sum $G_{-m}:=G(T^{-m})=\sum_{x\in \F_q}T^{-m}(x)\theta(x)$.
\end{lemma}

Since Lemma \ref{theta} holds only when the parameter is nonzero, we require that $a\neq 0$ and $b\neq 0$, or equivalently $j(E)\neq 0,1728$.
For $A$ we have
$$A=\frac{1}{q-1}\sum_i G_{-i} T^i(b) \sum_z T^i(z)=G_0=-1$$
where the second equality follows from the fact that the innermost sum is 0 unless $i=0$, at which it is $q-1$.
Similarly,
$$B=\frac{1}{(q-1)^2}\sum_{i,j}G_{-i}G_{-j}T^i(b)T^j(-1)\sum_z T^{i+j}(z) \sum_y T^{2j}(y)$$
and the inner sums here are nonzero only when $2j=0$ and $j=-i$. Plugging in these values gives
$$B=1+q T^{\frac{q-1}{2}}(b).$$
We simply expand $C$ (because it will cancel soon) to get
$$C=\frac{1}{(q-1)^3}\sum_{i,j,k}G_{-i}G_{-j}G_{-k}T^j(a)T^k(b)\sum_z T^{i+j+k}(z) \sum_x T^{3i+j}(x).$$
Finally, we expand D
\begin{align*}
D=&\frac{1}{(q-1)^4}\sum_{i,j,k,l}G_{-i}G_{-j}G_{-k}G_{-l}T^j(a)T^k(b)T^l(-1)\\
& \cdot \sum_z T^{i+j+k+l}(z)\sum_x T^{3i+j}(x) \sum_y T^{2l}(y).
\end{align*}
Again, the only nonzero terms occur when $l=0$ or $l=\frac{q-1}{2}$. The $l=0$ term is
$$\frac{1}{(q-1)^3}\sum_{i,j,k}G_{-i}G_{-j}G_{-k}G_{0}T^j(a)T^k(b)\sum_z T^{i+j+k}(z)\sum_x T^{3i+j}(x) $$
and since $G_0=-1$ this term cancels with the $C$ term in the expression for $q(|E(\F_q)|-1)$.
Assuming now that $l=\frac{q-1}{2}$, both inner sums will be nonzero only when $j=-3i$ and $k=\frac{q-1}{2}+2i$. We may write this term as
\begin{equation}\label{dterm}
D_{\frac{q-1}{2}}:=\frac{1}{q-1}\sum_{i}G_{-i} G_{3i} G_{-\frac{q-1}{2}-2i} G_{\frac{q-1}{2}} T^{-3i}(a) T^{\frac{q-1}{2}+2i}(b)T^{\frac{q-1}{2}}(-1)
\end{equation}
and we may reduce this equation further by noting that $q\equiv 1 \pmod{4}$ implies that $G_{\frac{q-1}{2}}=\sqrt{q}$ and $T^{\frac{q-1}{2}}(-1)=1$. Combining the above results yields the expression

\begin{equation}
q(|E(\F_q)|-1)=q^2+q\cdot T^{\frac{q-1}{2}}(b)+\frac{\sqrt{q}}{q-1}\sum_{i}G_{-i} G_{3i} G_{-\frac{q-1}{2}-2i}  T^{-3i}(a) T^{\frac{q-1}{2}+2i}(b)
\end{equation}

Now we expand $G_{3i}$ and $G_{-\frac{q-1}{2} -2i}=G_{-2(\frac{q-1}{4}+i)}$ using the Davenport-Hasse relation from \cite{lang}.
 \begin{theorem}[Davenport-Hasse Relation \cite{lang}]\label{hasdavgen} Let  $m$ be a positive integer and let $q=p^e$ be a prime power such that $q\equiv 1 \pmod{m}$. Let $\theta$ be the additive character on $\F_q$ defined by $\theta(\alpha)=\zeta^{\tr \alpha}$, where $\zeta=e^{2 \pi i /p}$. For multiplicative characters $\chi, \psi \in \widehat{\F}_q^*$ we have
$$\prod_{\chi^m=1}G(\chi \psi)=-G(\psi^m)\psi(m^{-m})\prod_{\chi^m=1}G(\chi).$$
\end{theorem}
The cases for $m=3$, $m=2$ may be restated as follows.
\begin{cor}[Davenport-Hasse for $q\equiv 1\pmod{3}$]\label{hasdav} If $k\in \Z$ and $q$ satisfies $q\equiv 1 \pmod{3}$ then $$G_kG_{k+\frac{q-1}{3}}G_{k+\frac{2(q-1)}{3}}=qT^{-k}(27)G_{3k}.$$ \end{cor}

\begin{cor}[Davenport-Hasse for $q\equiv 1 \pmod{2}$] If $k\in \Z$ and $q$ satisfies $q\equiv 1\pmod{2}$ then
$$G_{-k}G_{-\frac{q-1}{2}-k}=G_{-2k}T^k(4)G_{\frac{q-1}{2}}.$$
\end{cor}

We may then write
\begin{align*}
G_{3i}=&\frac{G_i G_{i+\frac{q-1}{3}} G_{i+\frac{2(q-1)}{3}} T^i(27)}{q}\\
G_{-\frac{q-1}{2}-2i}=&\frac{G_{-i-\frac{q-1}{4}}G_{-i-\frac{3(q-1)}{4}}}{G_{\frac{q-1}{2}}T^{i+\frac{q-1}{4}}(4)}.
\end{align*}
Plugging this in to equation \ref{dterm} gives
{\small
$$D_{\frac{q-1}{2}}=\frac{T^{\frac{q-1}{2}}(b)}{q(q-1)T^{\frac{q-1}{4}}(4)}\sum_i G_{-i}G_i G_{i+\frac{q-1}{3}} G_{i+\frac{2(q-1)}{3}} G_{-i-\frac{q-1}{4}}G_{-i-\frac{3(q-1)}{4}}T^i\left(\frac{27b^2}{4a^3}\right).$$}
In order to write $a(E(\F_q))$ as a finite field hypergeometric function, we use the fact that if $T^{m-n}\neq \epsilon$, then
\begin{equation}\label{jacgauss}
\left({ T^m \atop T^n } \right)=\frac{G_mG_{-n}T^n(-1)}{G_{m-n}q}.
\end{equation}
This is another way of stating the classical identity $G(\chi_1)G(\chi_2)=J(\chi_1,\chi_2)G(\chi_1\chi_2)$ which holds whenever $\chi_1\chi_2$ is a primitive character.

Now use equation \ref{jacgauss} to write
\begin{align}
G_{i+\frac{q-1}{3}}G_{-i-\frac{q-1}{4}}=&\bin{T^{i+\frac{q-1}{3}}}{T^{i+\frac{q-1}{4}}}G_{\frac{q-1}{12}}q T^{i+\frac{q-1}{4}}(-1) \label{hd1}\\
G_{i+\frac{2(q-1)}{3}}G_{-i-\frac{3(q-1)}{4}}=&\bin{T^{i+\frac{2(q-1)}{3}}}{T^{i+\frac{3(q-1)}{4}}}G_{-\frac{q-1}{12}}q T^{i+\frac{3(q-1)}{4}}(-1) \label{hd2}
\end{align}
and plugging  in equations \ref{hd1}, \ref{hd2} gives
{\small
$$D_{\frac{q-1}{2}}=\frac{T^{\frac{q-1}{2}}(b)q}{(q-1)T^{\frac{q-1}{4}}(4)}G_{\frac{q-1}{12}}G_{-\frac{q-1}{12}}
\sum_i G_i G_{-i} \bin{T^{i+\frac{q-1}{3}}}{T^{i+\frac{q-1}{4}}}\bin{T^{i+\frac{2(q-1)}{3}}}{T^{i+\frac{3(q-1)}{4}}}
T^i\left(\frac{27b^2}{4a^3}\right).$$}
Since $G_{\frac{q-1}{12}}G_{-\frac{q-1}{12}}=q T^{\frac{q-1}{12}}(-1)$ we may write
$$D_{\frac{q-1}{2}}=\frac{T^{\frac{q-1}{2}}(b)T^{\frac{q-1}{12}}(-1)q^2}{(q-1)T^{\frac{q-1}{4}}(4)}
\sum_i G_i G_{-i} \bin{T^{i+\frac{q-1}{3}}}{T^{i+\frac{q-1}{4}}}\bin{T^{i+\frac{2(q-1)}{3}}}{T^{i+\frac{3(q-1)}{4}}}
T^i\left(\frac{27b^2}{4a^3}\right).$$

Next, we eliminate the $G_iG_{-i}$ term. If $i\neq 0$ then $G_iG_{-i}=q T^i(-1)$, and if $i=0$ then $G_i G_{-i}=1=q T^i(-1) - (q-1)$. Plugging in the appropriate identities for each $i$ we may write equation \ref{dterm} as
\begin{align*}
D_{\frac{q-1}{2}}=&\frac{T^{\frac{q-1}{2}}(b)T^{\frac{q-1}{12}}(-1) q^3}{(q-1)T^{\frac{q-1}{4}}(4)}\sum_i \bin{T^{i+\frac{q-1}{3}}}{T^{i+\frac{q-1}{4}}}\bin{T^{i+\frac{2(q-1)}{3}}}{T^{i+\frac{3(q-1)}{4}}}
T^i\left(-\frac{27b^2}{4a^3}\right)\\
&- \frac{T^{\frac{q-1}{2}}(b)T^{\frac{q-1}{12}}(-1)q^2}{T^{\frac{q-1}{4}}(4)} \bin{T^{\frac{q-1}{3}}}{T^{\frac{q-1}{4}}}\bin{T^{\frac{2(q-1)}{3}}}{T^{\frac{3(q-1)}{4}}}.
\end{align*}
By equation \ref{jacgauss} we have
$$\bin{T^{\frac{q-1}{3}}}{T^{\frac{q-1}{4}}}\bin{T^{\frac{2(q-1)}{3}}}{T^{\frac{3(q-1)}{4}}}=
\frac{G_{\frac{q-1}{3}} G_{-\frac{q-1}{4}}G_{\frac{2(q-1)}{3}}G_{-\frac{3(q-1)}{4}}}{G_{\frac{q-1}{12}}G_{-\frac{q-1}{12}}q^2}=
\frac{T^{\frac{q-1}{3}}(-1)T^{\frac{q-1}{4}}(-1)}{qT^{\frac{q-1}{12}}(-1)}$$
and so the second term reduces to $- \frac{T^{\frac{q-1}{2}}(b)T^{\frac{q-1}{4}}(-1) q }{T^{\frac{q-1}{4}}(4)}$. Equation \ref{dterm} becomes
\begin{align*}
D_{\frac{q-1}{2}}=&\frac{T^{\frac{q-1}{2}}(b)T^{\frac{q-1}{12}}(-1) q^3}{(q-1)T^{\frac{q-1}{4}}(4)}\sum_i \bin{T^{i+\frac{q-1}{3}}}{T^{i+\frac{q-1}{4}}}\bin{T^{i+\frac{2(q-1)}{3}}}{T^{i+\frac{3(q-1)}{4}}}
T^i\left(-\frac{27b^2}{4a^3}\right)\\
&- \frac{T^{\frac{q-1}{2}}(b)T^{\frac{q-1}{4}}(-1) q }{T^{\frac{q-1}{4}}(4)}.
\end{align*}
Make the substitution $i \to i-\frac{q-1}{4}$ to get
\begin{align*}
D_{\frac{q-1}{2}}=&T^{\frac{q-1}{12}}(-1)q^2T^{\frac{q-1}{4}}\left(\frac{-a^3}{27}\right)
\cdot \frac{q}{q-1}\sum_i \bin{T^{i+\frac{q-1}{12}}}{T^{i}}\bin{T^{i+\frac{5(q-1)}{12}}}{T^{i+\frac{q-1}{2}}}T^{i}\left(-\frac{27b^2}{4a^3}\right)\\
&- \frac{T^{\frac{q-1}{2}}(b)T^{\frac{q-1}{4}}(-1) q }{T^{\frac{q-1}{4}}(4)}\\
=&T^{\frac{q-1}{12}}(-1)q^2T^{\frac{q-1}{4}}\left(\frac{-a^3}{27}\right)
\Fha{T^{\frac{q-1}{12}}}{ T^{\frac{5(q-1)}{12}}}{ T^{\frac{q-1}{2}} }{ -\frac{27b^2}{4a^3}}_q \\
&- \frac{T^{\frac{q-1}{2}}(b)T^{\frac{q-1}{4}}(-1) q }{T^{\frac{q-1}{4}}(4)}.
\end{align*}
Putting this all together then gives
\begin{align*}
q(|E(\F_q)|-1)=&q^2+q T^{\frac{q-1}{2}}(b)-\frac{T^{\frac{q-1}{2}}(b)T^{\frac{q-1}{4}}(-1) q }{T^{\frac{q-1}{4}}(4)}\\
& \ +T^{\frac{q-1}{12}}(-1)T^{\frac{q-1}{4}}\left(\frac{-a^3}{27}\right) q^2 \cdot
\Fha{ T^{\frac{q-1}{12}}}{ T^{\frac{5(q-1)}{12}}}{ T^{\frac{q-1}{2}}}{ -\frac{27b^2}{4a^3}}_q.
\end{align*}
Equivalently,
\begin{align*}
|E(\F_q)|=&q+1+T^{\frac{q-1}{2}}(b)\left(1-\frac{T^{\frac{q-1}{4}}(-1)}{T^{\frac{q-1}{4}}(4)}\right)\\
&+T^{\frac{q-1}{12}}(-1)T^{\frac{q-1}{4}}\left(\frac{-a^3}{27}\right) q\cdot
\Fha{T^{\frac{q-1}{12}}}{ T^{\frac{5(q-1)}{12}}}{T^{\frac{q-1}{2}}}{ -\frac{27b^2}{4a^3}}_q.
\end{align*}
Noting that $T^{\frac{q-1}{12}}(-1)T^{\frac{q-1}{4}}\left(\frac{-a^3}{27}\right)=T^{\frac{q-1}{4}}\left(\frac{a^3}{27}\right)$ and $T^{\frac{q-1}{4}}(-1)=T^{\frac{q-1}{2}}(2)$ (both depend only on the congruence of $q\pmod{8}$)
reduces the expression to
$$|E(\F_q)|=q+1+
T^{\frac{q-1}{4}}\left(\frac{a^3}{27}\right) q\cdot
\Fha{T^{\frac{q-1}{12}}}{ T^{\frac{5(q-1)}{12}}}{T^{\frac{q-1}{2}} }{ -\frac{27b^2}{4a^3}}_q.$$
Since $a(E(\F_q))=q+1-|E(\F_q)|$, we have proven that
$$a(E(\F_q))=-q \cdot T^{\frac{q-1}{4}}\left(\frac{a^3}{27}\right) \cdot
\Fha{T^{\frac{q-1}{12}}}{ T^{\frac{5(q-1)}{12}} }{ T^{\frac{q-1}{2}}  }{ -\frac{27b^2}{4a^3}}_q .$$
\end{proof}

\subsection{Hypergeometric Transformation Laws and the Proof of Theorem \ref{genweier}}

We now prove Theorem \ref{genweier} as a consequence of Theorem \ref{weier} and the following transformation laws found in \cite{greene} given here for the special case of $\Fh$ functions.

\begin{theorem}[\cite{greene}\label{oneminus} Theorem 4.4(i)] For characters $A,B,C$ of $\F_q$ and $x\in \F_q$, $x\neq 0,1$,
$$\Fha{A}{B}{C}{x}_q=A(-1)\cdot \Fha{A}{B}{AB\overline{C}}{1-x}_q.$$
\end{theorem}

\begin{theorem}[\cite{greene} Theorem 4.2(ii)] \label{oneover} For characters $A,B,C$ of $\F_q$ and $x\in \F_q^*$,
$$\Fha{A}{B}{C}{x}_q=ABC(-1)\overline{A}(x)\cdot \Fha{A}{A\overline{C}}{A\overline{B}}{\frac{1}{x}}_q.$$
\end{theorem}

\begin{proof}[Proof of Theorem \ref{genweier}]
 We begin by noting that we may apply  Theorem \ref{oneminus} to the the expression in Theorem \ref{weier} because the parameter $-\frac{27b^2}{4a^3}$ will equal 1 if and  only if the discriminant of $E$ is 0, which we exclude. Similarly, it will equal 0 if and only if  $b=0$, in which case $j=1728$, and we exclude this case as well. So we begin by applying Theorem \ref{oneminus} to obtain the expression
$$a(E(\F_q))=-qT^{\frac{q-1}{4}}\left(-\frac{a^3}{27}\right)\Fha{T^{\frac{q-1}{12}}}{T^{\frac{5(q-1)}{12}}}{\varepsilon}{\frac{4a^3+27b^2}{4a^3}}_q.$$
Applying Theorem \ref{oneover} to this then gives
{\footnotesize
\begin{align*}
a(E(\F_q))=& -q \cdot T^{\frac{q-1}{4}}\left(\frac{-a^3}{27}\right)T^{\frac{q-1}{12}}\left(\frac{4a^3}{4a^3+27b^2}\right)\Fha{T^{\frac{q-1}{12}}}{T^{\frac{q-1}{12}}}{T^{\frac{2(q-1)}{3}}}{\frac{4a^3}{4a^3+27b^2}}_q\\
=&-q\cdot T^{\frac{q-1}{12}}\left(\frac{-4a^{12}}{3^9(4a^3+27b^2)}\right)\Fha{T^{\frac{q-1}{12}}}{T^{\frac{q-1}{12}}}{T^{\frac{2(q-1)}{3}}}{\frac{4a^3}{4a^3+27b^2}}_q\\
=&-q\cdot T^{\frac{q-1}{12}}\left(\frac{a^{12}}{3^{12}}\cdot \frac{4^3 3^3}{-16(4a^3+27b^2)}\right)\Fha{T^{\frac{q-1}{12}}}{T^{\frac{q-1}{12}}}{T^{\frac{2(q-1)}{3}}}{\frac{4a^3}{4a^3+27b^2}}_q\\
=&-q\cdot T^{\frac{q-1}{12}}\left(\frac{1728}{-16(4a^3+27b^2)}\right)\Fha{T^{\frac{q-1}{12}}}{T^{\frac{q-1}{12}}}{T^{\frac{2(q-1)}{3}}}{\frac{4a^3}{4a^3+27b^2}}_q\\
=&-q\cdot  T^{\frac{q-1}{12}}\left(\frac{1728}{\Delta(E)}\right)\Fha{T^{\frac{q-1}{12}}}{T^{\frac{q-1}{12}}}{T^{\frac{2(q-1)}{3}}}{\frac{j(E)}{1728}}_q
\end{align*}}
where $\Delta(E)=-16(4a^3+27b^2)$ is the discriminant of $E$ and $j(E)=\frac{1728\cdot 4a^3}{4a^3+27b^2}$ is the $j$-invariant of $E$.

\end{proof}

\bibliographystyle{amsplain}

\end{document}